\documentclass{amsart}

\usepackage{amsmath,amssymb,verbatim}
\usepackage{amsthm}
\usepackage{enumerate}
\usepackage{url}
\usepackage{mathtools}
\usepackage{color}

\newtheorem{theorem}{Theorem}[section]
\newtheorem{lemma}[theorem]{Lemma}

\newtheorem{proposition}[theorem]{Proposition}

\newtheorem{fact}[theorem]{Fact}

\theoremstyle{definition}
\newtheorem{definition}[theorem]{Definition}

\newtheorem{remark}[theorem]{Remark}

\def\tp{\operatorname{tp}}

\def\cL{\mathcal L}
\def\cR{\mathcal R}
\def\cU{\mathcal U}

\def\R{\mathbb R}
\def\N{\mathbb N}
\def\Z{\mathbb Z}
\def\Q{\mathbb Q}
\def\cbar{\bar{c}}
\def\ybar{\bar{y}}
\def\Th{\textnormal{Th}}
\def\st{\textnormal{st}}
\newcommand{\seq}{\subseteq}
\newcommand{\inv}{^{\text{-}1}}
\def\bbar{\bar{b}}
\def\ybar{\bar{y}}
\def\zbar{\bar{z}}

\newcommand{\claimqed}{\hfill$\dashv_{\text{\scriptsize{claim}}}$}

\renewcommand{\phi}{\varphi}
\AtBeginDocument{%
   \def\MR#1{}
}

\title{Pseudofinite proofs of the stable graph regularity lemma}

\author{G. Conant}
\address{Department of Mathematics\\
The Ohio State University\\
Columbus, OH, 43210, USA}
\email{conant.38@osu.edu}

\author{C. Terry}
\address{Department of Mathematics\\
The Ohio State University\\
Columbus, OH, 43210, USA}
\email{terry.376@osu.edu}

\date{August 31, 2023}

\begin{document}

\maketitle

\begin{abstract}
In this expository article, we give a detailed proof of a qualitative version of the Mallaris-Shelah regularity lemma for stable graphs \cite{MaShStab} using only basic  local stability theory  and an ultraproduct construction. This proof strategy was first established by Malliaris and Pillay \cite{MallPi}, and later simplified by Pillay in \cite{PiDR}. We provide some further simplifications, and also explain  how the pseudofinite approach can be used to obtain a qualitative  strengthening (in comparison to \cite{MallPi,PiDR}) in terms of ``functional error". To illustrate the extra leverage obtained by functional error, we give an elementary argument for extracting equipartitions from arbitrary partitions.
\end{abstract}

\section{Introduction}

In \cite{MaShStab}, Malliaris and Shelah  prove a strengthened version of Szemer\'{e}di's regularity lemma for the special class of stable finite graphs. Looking back over ten years later, one can see the deep impact this result has had on the flourishing field of interactions between model theory and finite combinatorics (e.g., \cite{ChSt, ChStNIP, CPT, CPTNIP, TeWo, TeWo2}).  To state their result, we first recall that a graph $(V,E)$ is \emph{$k$-stable} if there do not exist $v_1,\ldots,v_k,w_1,\ldots,w_k\in V$ such that $E(v_i,w_j)$ holds if and only if $i\leq j$. If $V$ is finite and  $X,Y\seq V$ are nonempty, then we define the edge density $d(X,Y)$ to be $|E\cap (X\times Y)|/|X\times Y|$ (viewing $E$ as a subset of $V\times V$). 

\begin{theorem}[Malliaris-Shelah]\label{thm:MS} 
Fix $k\geq 1$ and $\epsilon>0$. Then for any finite $k$-stable graph $(V,E)$, there is a partition $V=Y_1\cup\ldots\cup Y_n$, with $n\leq O_k((1/\epsilon)^{O_k(1)})$, such that for all $i,j\geq 1$, $||Y_i|-|Y_j||\leq 1$ and either $d(Y_i,Y_j)<\epsilon$ or $d(Y_i,Y_j)>1-\epsilon$.
\end{theorem}

There are several key differences between Theorem \ref{thm:MS} and Szemer\'{e}di's regularity lemma. First, one obtains a polynomial (in $1/\epsilon$) bound on the size of the paritition, rather than the tower-type bounds necessary for arbitrary finite graphs \cite{GowSRL}. Second, \emph{all pairs} in the partition behave regularly, while for arbitrary graphs a small number of irregular pairs must be allowed. Finally, the density of each regular pair is close to $0$ or $1$  (we refer to this as ``almost homogeneity"), rather than behaving like random bipartite graphs of some arbitrary density in $[0,1]$. 

Shortly after \cite{MaShStab}, Malliaris and Pillay \cite{MallPi} proved a similar regularity lemma for stable formulas in arbitrary first-order structures with respect to any Keisler measure. When applying this result to a pseudofinite structure (with the normalized pseudofinite counting measure) one obtains a version of Theorem \ref{thm:MS} without equipartitions or quantitative bounds. The purpose of this article is to revisit the pseudofinite approach to stable graph regularity. We will give the complete proof, starting from the basic preliminaries in local stability. Our proof simplifies several model-theoretic and topological steps in some ways (motivated in part by recent work on stable  regularity for functions  by the first author, Chavarria, and Pillay \cite{CCP}; see also \cite{PiDR}). We also take a different combinatorial approach based on a ``symmetry lemma" for pairs of good sets (see Lemma \ref{lem:twosticks}). Finally, we will incorporate the regime of controlling irregularity with an arbitrary function of the size of the partition, rather than a constant $\epsilon>0$. In addition to providing a sharper qualitative statement, the functional error also allows one to obtain equipartitions in a  quick and straightforward way (compared to \cite{MaShStab}, which obtains equiparititions via a probabilistic argument).  We now state precisely what we will prove.

\begin{theorem}\label{thm:maindown}
Fix $k\geq 1$, $\epsilon>0$, and $\sigma\colon \N\to (0,1)$. Then for any finite $k$-stable graph $(V,E)$, there is a partition $V=Y_0\cup Y_1\cup\ldots\cup Y_n$, with $n\leq O_{k,\epsilon,\sigma}(1)$, such that $|Y_0|\leq \epsilon|V|$ and, for all $i,j\geq 1$, $|Y_i|=|Y_j|$ and either $d(Y_i,Y_j)< \sigma(n)$ or $d(Y_i,Y_j)> 1-\sigma(n)$.
\end{theorem}

Let us compare and contrast the two theorems above. First, it is important to note that the pseudofinite approach to proving asymptotic results about finite graphs \emph{provides no quantitative information on bounds}. This is a significant drawback, since the polynomial bound obtained in Theorem \ref{thm:MS} is one of the hallmark features of so-called ``tame regularity". Thus we emphasize that the focus of this article is \emph{entirely qualitative}. A quantitative proof of Theorem \ref{thm:maindown} is given by the second author and Wolf in \cite{TWgraphs}.

If $\sigma$ is a constant error function (e.g., $\sigma(n)=\epsilon$ for all $n\in\N$), then Theorem \ref{thm:maindown} is qualitatively  equivalent to Theorem \ref{thm:MS}. Indeed, Theorem \ref{thm:MS} implies  Theorem \ref{thm:maindown} (with constant $\sigma$) since one can obtain $|Y_i|=|Y_j|$ for $i,j\geq 1$, rather than $||Y_i|-|Y_j||<1$, by choosing $Y_0$ of size at most $n$. Conversely, given Theorem \ref{thm:maindown}, one can distribute $Y_0$ among the sets $Y_1,\ldots,Y_n$ as evenly as possible in order to (qualitatively) deduce  Theorem \ref{thm:MS} (this involves small calculation and change in $\epsilon$, which we leave to the reader). However, in order to formulate Theorem \ref{thm:maindown} with an error \emph{function} as we have done, the exceptional set $Y_0$ of constant error size  seems to be necessary.

The outline of this article is as follows. Section \ref{sec:prelim} provides complete proofs of the necessary ingredients from local stability theory, starting from ``first principles" in model theory. In particular, we prove the characterization of stable formulas in terms of definability of types, and that any local Keisler measure obtained from a stable formula can be written as a countable weighted sum of types. In Section \ref{sec:main}, we prove Theorem \ref{thm:maindown}. We will aim to provide all of the necessary details for those readers having no previous experience with combinatorial proofs via an ultraproduct construction. For example, we will provide complete details on formulating the pseudofinite normalized counting measure with a first-order language (Section \ref{sec:pcm}). In Section \ref{sec:more}, we make a few further remarks. In particular, as our approach to Theorem \ref{thm:maindown} differs  from that of Malliaris and Pillay  \cite{MallPi}, we provide the details of their result. Finally, we discuss (without any proofs) the corresponding stable \emph{arithmetic} regularity lemma for subsets of groups. The motivation for this discussion is that in the group setting the exceptional set $Y_0$ disappears and, moreover, one can provide a concrete model-theoretic explanation for why this occurs. 

This article is based on (but not exactly the same as) two lectures given by the first author at the Thematic Program on Trends in Pure and Applied Model Theory, held at the Fields Institute in the fall of 2021. Several of the key combinatorial insights are taken from a graduate course taught by the second author at the University of Chicago in the fall of 2019. 

\subsection*{Acknowledgements}
Both authors thank the Fields Institute for Research in Mathematical Sciences for their support and hospitality. Thanks also to Maryanthe Malliaris for many helpful discussions about stable regularity.  Conant was partially supported by NSF grant DMS-2204787, as well as a Simons Distinguished Visitorship held by Anand Pillay. Terry was partially supported by NSF grant DMS-2115518.

\section{Preliminaries from local stability}\label{sec:prelim}

Throughout this section, we fix a complete first-order theory $T$ with language $\cL$. Since we will be focusing on the ``local theory" of a stable formula, we may assume $\cL$ is countable. We also assume all models of $T$ are infinite. We start with a  review of the basic terminology and definitions regarding local type spaces and stability. 

 Let $\phi(x,y)$ be a fixed $\cL$-formula whose free variables have been partitioned into two tuples $x$ (the ``object variables") and $y$ (the ``parameter variables"). Let $\phi^*(y,x)$ denote the ``dual formula", which is identical to $\phi(x,y)$, but comes with the assignments of object and parameter variables exchanged. 

Fix $M\models T$. A \emph{$\phi$-formula over $M$} is a finite Boolean combination of instances $\phi(x,b)$ where $b\in M^y$. A \emph{complete $\phi$-type over $M$} is a maximal consistent set of $\phi$-formulas over $M$. We let $S_\phi(M)$ denote the collection of complete $\phi$-types over $M$. Note that any $p\in S_\phi(M)$ is  determined by its behavior on instances $\phi(x,b)$ for $b\in M^y$. We say that $p$ is \emph{definable (over $M$)} if this behavior is given by a formula, i.e.,  if there is an $\cL_M$-formula $\psi(y)$ such that, for any $b\in M^y$, $\phi(x,b)\in p$ if and only if $M\models\psi(b)$. Finally recall  that $S_\phi(M)$ is a \emph{Stone space} (i.e., totally disconnected, compact, and Hausdorff)  where the basic clopen sets are of the form
\[
[\psi(x)]\coloneqq \{p\in S_\phi(M):\psi(x)\in p\}
\]
for some $\phi$-formula $\psi(x)$ over $M$. 

Given an integer $k\geq 1$, we say that $\phi(x,y)$ is \emph{$k$-stable in $T$} if there is no $\phi$-order of length $k$ in some/any model of $T$, i.e., 
\[
T\models\neg\exists x_1\ldots x_k,y_1\ldots y_k\bigwedge_{i,j\leq k}(\phi(x_i,y_j)\leftrightarrow i\leq j)
\]
(note that we are abusing syntax in the above first-order sentence for sake a readability).
The formula $\phi(x,y)$ is \emph{stable in $T$} if it is $k$-stable in $T$ for some $k\geq 1$.

Next we prove a fundamental result about the relationship between stability, counting types, and definability of types. This theorem was first proved by Shelah, although the argument we give is due to Hrushovski and Pillay (see the remarks following the proof for further clarification).

\begin{theorem}\label{thm:stabchar}
Given a formula $\phi(x,y)$, the following are equivalent.
\begin{enumerate}
\item[$(i)$] $\phi(x,y)$ is stable in $T$. 
\item[$(ii)$] For any $M\models T$, $|S_\phi(M)|\leq |M|$. 
\item[$(iii)$] For any $M\models T$, every type $p\in S_\phi(M)$ is definable. In particular, there is some $\phi^*$-formula $\psi(y)$ over $M$ such that, given $b\in M^y$, we have $\phi(x,b)\in p$ if and only if $M\models\psi(b)$.
\end{enumerate}
\end{theorem}
\begin{proof}
$(iii)\Rightarrow (ii)$. If $(iii)$ holds then each $p\in S_\phi(M)$ is completely determined by the formula $\psi(y)$, and there are at most $|M|$ such formulas.

$(ii)\Rightarrow (i)$.  Suppose $(i)$ fails. By compactness, there is a model $N\models T$ containing $\R$-indexed sequences $(a_i)_{i\in \R}$ and $(b_i)_{i\in\R}$ such that $N\models\phi(a_i,b_j)$ if and only if $i\leq j$. Let $M\prec N$ be a countable submodel containing $(b_i)_{i\in\Q}$. Then for $i<j$ from $\R$, we have $\tp_\phi(a_i/M)\neq \tp_\phi(a_j/M)$ since $N\models\phi(a_i,b_k)\wedge\neg\phi(a_j,b_k)$ for any $k\in\Q$ with $i<k<j$.  So $|S_\phi(M)|>|M|$.  

$(i)\Rightarrow (iii)$. Assume $\phi(x,y)$ is $k$-stable in $T$.  Fix $M\models T$ and $p\in S_\phi(M)$. We will construct elements $a_1,\ldots,a_{2k}\in M^x$ such that, for any $b\in M^y$, $\phi(x,b)\in p$ if and only $M\models \phi(a_t,b)$  for at least $k$ indices $t\in [2k]$.\footnote{Given an integer $n\geq 1$, we let $[n]=\{1,\ldots,n\}$.} In particular, in the statement of $(iii)$ we can take $\psi(y)$ to be the formula
\[
\bigvee_{X\in \binom{[2k]}{k}}\bigwedge_{t\in X}\phi(a_t,y).
\]

We start the construction with an arbitrary $a_1\in M^x$. Suppose we have constructed $a_1,\ldots,a_n$ for some fixed $n<2k$. Let $I_n$ denote the collection of sets $X\seq [n]$ such that, for some $b\in M^y$, we have $\neg\phi(x,b)\in p$ and  $M\models \phi(a_t,b)$ for all $t\in X$. For each $X\in I_n$, fix some $b^X_n\in M^y$ witnessing that $X$ satisfies the definition of $I_n$.  Dually, let $J_n$ denote the collection of sets $Y\seq [n]$ such that, for some $c\in M^y$, we have $\phi(x,c)\in p$ and  $M\models \neg\phi(a_t,c)$ for all $t\in Y$. For each $Y\in J_n$, fix a witness $c^Y_n$. We let $D$ denote the set of all such witnesses from  all previous stages $m\leq n$, i.e.,
\[
D=\{b^X_m:m\leq n,~X\in I_m\}\cup\{c^Y_m:m\leq n,~Y\in J_m\}.
\]
Note that $D$ is a finite subset of $M^y$. Since $p$ is finitely satisfiable in $M$, we may choose some $a_{n+1}\in M^x$ such that, for all $d\in D$, $M\models \phi(a_{n+1},d)$ if and only if $\phi(x,d)\in p$. By construction, we have the following properties:
\begin{enumerate}
\item[$(a)$] If $m\leq n$ and $X\in I_m$ then $M\models \neg\phi(a_{n+1},b^X_m)$.
\item[$(b)$] If $m\leq n$ and $Y\in J_m$ then $M\models\phi(a_{n+1},c^Y_n)$.
\end{enumerate}
This finishes the construction of $a_1,\ldots,a_{2k}$. 

Now fix $b\in M^y$. We want to show that $\phi(x,b)\in p$ if and only if $M\models\phi(a_t,b)$ for at least $k$ indices $t\in [2k]$. 

First, suppose there are $t_1<\ldots<t_k$ from $[2k]$ such that $M\models \phi(a_{t_i},b)$ for all $i\in [k]$. Toward a contradiction, suppose $\neg\phi(x,b)\in p$. Then for each $j\in [k]$, the set $X_j=\{t_1,\ldots,t_j\}$ is in $I_{t_j}$, and so we have the witness $b_j\coloneqq b^{X_j}_{t_j}$. If $i\leq j$ then $t_i\in X_j$, and so $M\models \phi(a_{t_i},b_j)$ by choice of $b_j$. On the other hand, if $i>j$ then we have $t_j\leq t_i-1$ and $X_j\in I_{t_j}$, which implies $M\models\neg \phi(a_{t_i},b_j)$ (apply $(a)$ above with $m=t_j$ and $n=t_i-1$). Altogether, given $i,j\in [k]$, we have $M\models\phi(a_{t_i},b_j)$ if and only if $i\leq j$, which contradicts $k$-stability of $\phi(x,y)$. 

Finally, suppose $M\models\phi(a_t,b)$ for at most $k-1$ indices $t\in [2k]$. Then there are $t_1<\ldots<t_{k+1}$ from $[2k]$ such that $M\models\neg\phi(a_{t_i},b)$ for all $i\in [k+1]$. Toward a contradiction, suppose $\phi(x,b)\in p$. Using an argument analogous to the previous case, one constructs $c_1,\ldots,c_{k+1}$ such that $M\models\neg\phi(a_{t_i},c_j)$ if and only if $i\leq j$. For $i\in [k]$, set $a'_i=a_{t_{k+2-i}}$ and $c'_i=c_{k+1-i}$. Then $M\models\phi(a'_i,c'_j)$ if and only if $i\leq j$, contradicting the choice of $k$. 
\end{proof}

\begin{remark}
The previous result is a special case of the Unstable Formula Theorem, first proved by Shelah \cite[Theorem II.2.2]{Shbook}. Our formulation of this theorem focuses on $\phi$-types over models, rather than over arbitrary  sets, in which case the proof is simpler (see the next remark for further discussion). The proof above (more specifically, the proof of $(i)\Rightarrow (iii)$) is due to Hrushovski and Pillay; see  \cite[Lemma 2.2]{Pibook} and/or \cite[Lemma 5.4]{HrPiGLF}. 
In part $(iii)$, the definition $\psi(y)$ for $p$ is a positive Boolean combination of instances $\phi(a,y)$ with $a\in M^x$. Moreover, the Boolean form is uniform across all types and models. In particular, if $\phi(x,y)$ is $k$-stable in $T$ and $\psi(y,x_1,\ldots,x_{2k})$ denotes the formula $\bigvee_{X\in \binom{[2k]}{k}}\bigwedge_{t\in X}\phi(x_t,y)$, then for any $M\models T$ and $p\in S_\phi(M)$ there are $a_1,\ldots,a_{2k}$ such that, for any $b\in M^y$, $\phi(x,b)\in p$ if and only if $M\models\psi(b,\bar{a})$. 
\end{remark}

\begin{remark}
If $\phi(x,y)$ is stable in $T$ then definability of types actually holds over arbitrary sets of parameters. In particular, for any $M\models T$ and $A\seq M$, any type $p\in S_\phi(A)$ is definable over $A$. However, note  the proof of $(i)\Rightarrow(iii)$ uses the fact that $M$ is a model when constructing $a_{n+1}$, and thus the argument does not generalize to types over arbitrary sets. Instead, the proof over arbitrary sets uses the characterization of stability of $\phi(x,y)$ in terms of a finite bound on the length of a certain binary tree built using instances of $\phi(x,y)$. This characterization was first proved by Shelah \cite[Chapter II]{Shbook} using infinitary methods, and then later Hodges \cite{Hodtrees} gave a finitary combinatorial proof with an explicit quantitative relationship between the stability of  $\phi(x,y)$ and the height of the binary tree.
\end{remark}

\begin{remark}
It follows from the proof of Theorem \ref{thm:stabchar} that to check stability of $\varphi(x,y)$ in $T$, it suffices to establish that $(ii)$ or $(iii)$ holds for any \emph{countable} $M\models T$.
\end{remark}

We now turn to measures.

\begin{definition}
Given $M\models T$, a \textbf{Keisler $\phi$-measure over $M$} is a finitely additive probability measure on the Boolean algebra of $\phi$-formulas over $M$, i.e., a function $\mu$ from $\phi$-formulas over $M$ to the interval $[0,1]$ satisfying the following axioms:
\begin{enumerate}
\item[$\ast$] $\mu(\theta(x))=0$ if $\theta(x)$ is inconsistent,
\item[$\ast$] $\mu(\neg\theta(x))=1-\mu(\theta(x))$, 
\item[$\ast$] $\mu(\theta_1(x)\vee\theta_2(x))=\mu(\theta_1(x))+\mu(\theta_2(x))-\mu(\theta_1(x)\wedge \theta_2(x))$.
\end{enumerate}
\end{definition}

Note  that if two $\phi$-formulas over $M$ are equivalent (modulo $T$ with constants for $M$), then any Keisler $\phi$-measure will assign them the same value. Thus one can also view a Keisler $\phi$-measure as a map on $\phi$-definable subsets of $M^x$. 

Recall that the space of ultrafilters (or types) over a Boolean algebra is a Stone space. Moreover any Stone space $X$ can be represented as the space of ultrafilters over the Boolean algebra of clopen sets in $X$. Via this correspondence (also called ``Stone duality"), one obtains a connection between finitely additive probability measures on Boolean algebras and regular Borel probability measures on Stone spaces. We state this connection now in the setting of Keisler measures. A proof can be found in \cite[Section 7.1]{Sibook}; see also \cite[Proposition 416Q]{Fremv4}.

\begin{fact}
Given a Keisler $\phi$-measure $\mu$ over $M\models T$, there is a unique regular Borel probability measure $\hat{\mu}$ on $S_\phi(M)$ such that given a $\phi$-formula $\theta(x)$ over $M$, 
\[
\mu(\theta(x))=\hat{\mu}([\theta(x)]).
\]
Moreover, the map $\mu\mapsto \hat{\mu}$ is a bijection between Keisler $\phi$-measures over $M$ and regular Borel probability measures on $S_\phi(M)$.
\end{fact}

In light of the above fact, we identify $\mu$ and $\hat{\mu}$ when there is no possibility for confusion. We also write $\mu(p)$ for $\mu(\{p\})$, and we identify a type $p\in S_\phi(M)$ with its associated Dirac measure on $S_\phi(M)$. 

\begin{remark}\label{rem:meas-of-type}
Let $\mu$ be a Keisler $\phi$-measure over $M$. Then by regularity of $\mu$, and since $S_\phi(M)$ is a Stone space, it follows that for any closed set $C\seq S_\phi(M)$, we have
\[
\mu(C)=\inf\{\mu(K):K\seq S_\phi(M)\text{ clopen containing $C$}\}.
\]
We will primarily apply this when $C$ is a singleton $\{p\}$.
\end{remark}

Next we prove a decomposition theorem for Keisler measures obtained from stable formulas (originally due to Keisler \cite{Keis}). In particular, such a measure can be written as a countable weighted sum of of types (i.e, Dirac measures). While this result holds for measures over any model of $T$, we point out in Remark \ref{rem:shortKM} below that the proof for countable models is much shorter and, moreover, suffices to obtain  the stable graph regularity lemma (Theorem \ref{thm:maindown}).

\begin{theorem}\label{thm:KMstab}
Suppose $\phi(x,y)$ is stable and let $\mu$ be a Keisler $\phi$-measure over $M\models T$. Define $D=\{p\in S_\phi(M):\mu(p)>0\}$ (the set of ``point masses" for $\mu$). Then $D$ is countable and $\mu(D)=1$. Thus $\mu$ can be written as a countable weighted sum of types, namely, $\mu=\sum_{p\in D}\mu(p)p$. 
\end{theorem}
\begin{proof}
First note that by additivity of $\mu$, for any $\epsilon>0$ there are at most $1/\epsilon$ distinct types $p\in S_\phi(M)$ with $\mu(p)\geq\epsilon$. Taking a union over rational $\epsilon$, we see that $D$ is countable. So, in particular, $D$ is Borel. To show that $\mu(D)=1$, it suffices to prove the following claim. \medskip

\noindent\emph{Claim.} Suppose $X\seq S_\phi(M)$ is Borel and $\mu(X)>0$. Then $X\cap D\neq\emptyset$.

\noindent\emph{Proof.} By regularity of $\mu$, it suffices to assume $X$ is closed. Toward a contradiction, suppose $X\cap D=\emptyset$. We will construct a tree-indexed sequence $(K_\eta)_{\eta\in 2^{<\omega}}$ of clopen sets such that for all $\eta\in 2^{<\omega}$,  $\mu(K_\eta\cap X)>0 $, $K_{\eta0}\seq K_\eta$, and $K_{\eta 1}=K_\eta\backslash K_{\eta0}$. 

Let $K_\emptyset=S_\phi(M)$ (so $\mu(K_\emptyset\cap X)>0$ by assumption). Fix $\eta\in 2^{<\omega}$ and suppose we have constructed $K_\eta$. For any $p\in X$, we have $\mu(p)=0$ since $X\cap D=\emptyset$, and thus by Remark \ref{rem:meas-of-type} there is a clopen set $K_p$ containing $p$ with $\mu(K_p)<\mu(K_\eta\cap X)$. If $p\in K_\eta\cap X$, then we may further assume $K_p\seq K_\eta$.  Note that $K_\eta\cap X$ is closed, and hence compact. Thus there are finitely many $p_1,\ldots,p_n\in K_\eta\cap X$ such that $K_\eta\cap X\seq\bigcup_{t=1}^n K_{p_t}$. By additivity of $\mu$, there is some $t$ such that $\mu(K_{p_t}\cap X)>0$. Thus we can set $K_{\eta 0}=K_{p_t}$ and $K_{\eta1}=K_\eta\backslash K_{\eta0}$ (note that $\mu(K_{\eta1}\cap X)>0$ since $\mu(K_{\eta0})<\mu(K_\eta\cap X)$). 

Now choose a sequence $(\psi_\eta(x))_{\eta\in 2^{<\omega}}$ of $\phi$-formulas over $M$ such that $K_\eta=[\psi_\eta(x)]$. Then for any  $\sigma\in 2^\omega$, the partial $\phi$-type $\pi_\sigma=\{\psi_\eta(x):\eta\sqsubseteq\sigma\}$ is consistent, and if $\sigma\neq\sigma'$ then $\pi_\sigma$ and $\pi_{\sigma'}$ are mutually inconsistent. Choose a countable submodel $M_0\preceq M$ such that each $\psi_\eta(x)$ is over $M_0$. Then each $\pi_\sigma$ extends to a  type $p_\sigma\in S_\phi(M_0)$, which yields $|S_\phi(M_0)|>|M_0|$. By Theorem \ref{thm:stabchar}, this contradicts stability of $\phi(x,y)$.  
\end{proof}

\begin{remark}
The previous proof is based on a research note by Gannon \cite{GanStab}, which establishes a connection between stable Keisler measures and the Sobjeck-Hammer decomposition theorem from  the theory of charges.
The proof of the ``Claim" can be made a little shorter through the use of Cantor-Bendixson rank. However, to show that  $S_\phi(M)$ has a defined Cantor-Bendixson rank (assuming stability of $\phi(x,y)$), one usually carries out a similar construction of uncountably many $\phi$-types over a countable model via an infinite binary tree.  
\end{remark}

\begin{remark}\label{rem:shortKM}
When applying Theorem \ref{thm:KMstab} to prove stable graph regularity, it suffices to assume $M$ is countable (via Lowenheim-Skolem). In this case, the proof of  Theorem \ref{thm:KMstab} is much simpler. Indeed, if $M$ is countable then so is $S_\phi(M)$ by Theorem \ref{thm:stabchar}. Therefore any set $X\seq S_\phi(M)$ is Borel and, by countable additivity, $\mu(X)=\mu(X\cap D)=\sum_{p\in X\cap D}\mu(p)$.  
\end{remark}

\section{Stable graph regularity}\label{sec:main}

\subsection{Notation}
Given $m\in\N$ and a set $S$ of parameters, we write $m\leq O_S(1)$ to mean that $m$ is bounded by a constant depending only on the parameters in $S$.

Given a graph $(V,E)$, a set $X\seq V$, and a vertex $b\in V$, we let $E(X,b)$ denote the edge neighborhood of $b$, i.e., the set of $a\in X$ such that $E(a,b)$ holds. We  use $\neg E$ to denote the complement of $E$ (in $V\times V$), and the notation $\neg E(X,b)$ is defined in the same way.

Similarly, given an $\cL$-structure $M$, an $\cL$-formula $\phi(x,b)$ over $M$, and a subset $X\seq M^x$ (usually definable), we let $\phi(X,b)$ denote the set of $a\in X$ such that $M\models\phi(a,b)$. We frequently identify $\cL$-formulas over $M$ with definable sets in $M$. For example, given an $\cL$-formula $\phi(x,y)$, a $\phi$-type $p\in S_\phi(M)$, and a $\phi$-definable subset $X\seq M^x$, we will write ``$X\in p$" to mean that $p$ contains a $\phi$-formula over $M$ defining $X$.

\subsection{Good sets and the Symmetry Lemma}

The proof of Theorem \ref{thm:maindown} will combine an ultraproduct construction together with a model-theoretic result about a single (infinite) graph. This latter result will be formulated using  the notion of ``good sets", as defined by Malliaris and Shelah \cite{MaShStab}. We now recall that definition, formulated in the setting of an arbitrary measure.

\begin{definition}
Let $(V,E)$ be a graph, and suppose $\mu$ is a finitely additive probability measure on some Boolean algebra of subsets of $V$ containing $E(V,b)$ for all $b\in V$. Fix a $\mu$-measurable set $X\seq V$ and some $\epsilon>0$. Then $X$ is \textbf{$\epsilon$-good in $(V,E)$ with respect to $\mu$} if for all $b\in V$ either $\mu(E(X,b))<\epsilon\mu(X)$ or $\mu(E(X,b))>(1-\epsilon)\mu(X)$.
\end{definition}

In the previous definition, if $V$ is finite and $\mu$ is the normalized counting measure on $V$, then we will omit the clause ``with respect to $\mu$". This setting recovers the notion of good sets as defined in \cite{MaShStab}.  Next we show that in finite graphs, pairs of good sets are automatically almost homogeneous. This is a folklore  result, which is referred to as a ``symmetry lemma" in recent work of the second author and Wolf \cite[Lemma 5.9]{TWgraphs}. See also \cite[Lemma 27]{BBW} for a combinatorial result based on the same underlying principle.

\begin{proposition}\label{prop:twosticks}
Let $(V,E)$ be a finite graph. Fix $X,Y\seq V$ and $\alpha,\beta,\delta,\epsilon>0$. Define
 \begin{align*}
X_0 &= \{a\in X:|E(a,Y)|<\delta|Y|\},\text{ and}\\
Y_1 & = \{b\in Y:|E(X,b)|>(1-\epsilon)|X|\}
\end{align*}
Assume $|Y_1|\geq\alpha|Y|$ and $\delta(1-\beta)\leq\alpha(\beta-\epsilon)$. Then either $|X_0|<\beta|X|$ or $|X_0|>(1-\beta)|X|$.
\end{proposition}
\begin{proof}
For a contradiction, suppose $\beta|X|\leq |X_0|\leq (1-\beta)|X|$. If $b\in Y_1$ then, since $|X_0|\geq\beta|X|$ and 
\[
|\neg E(X_0,b)|\leq |\neg E(X,b)|<\epsilon|X|,
\]
it follows that $|E(X_0,b)|>(\beta-\epsilon)|X|$. Therefore 
\[
|E(X_0,Y_1)|=\sum_{b\in Y_1}|E(X_0,b)|>(\beta-\epsilon)|X||Y_1|\geq \alpha(\beta-\epsilon)|X||Y|.
\]
On the other hand, since $|X_0|\leq (1-\beta)|X|$, 
\[
|E(X_0,Y_1)|\leq |E(X_0,Y)|\leq \sum_{a\in X_0}|E(a,Y)|<\delta|X_0||Y|\leq \delta(1-\nu)|X|.
\]
So we conclude $\alpha(\beta-\epsilon)<\delta(1-\beta)$, which is a contradiction. 
\end{proof}

\begin{lemma}[Symmetry Lemma]\label{lem:twosticks}
Let $(V,E)$ be a finite graph, and suppose $X,Y\seq V$ are both $\epsilon^2/4$-good for $(V,E)$. Then either $d(X,Y)<\epsilon$ or $d(X,Y)>1-\epsilon$.
\end{lemma}
\begin{proof}
Set $\delta=\epsilon/2$. Without loss of generality, we may assume $\epsilon\leq 1$, and thus $\delta^2+\delta\leq\epsilon$. Define the sets
\begin{align*}
X_0 &=\{a\in X:|E(a,Y)|<\delta^2|Y|\},\text{ and}\\
Y_1 &=\{b\in Y:|E(X,b)|>(1-\delta^2)|X|\}.
\end{align*}
Since $X$ and $Y$ are $\delta^2$-good, $|E(X,b)|<\delta^2|X|$ for any $b\in Y_0\coloneqq Y\backslash Y_1$, and $|E(a,Y)|>(1-\delta^2)|Y|$ for any $a\in X_1\coloneqq X\backslash X_0$. If $|Y_1|<\delta|Y|$ then 
\[
|E(X,Y)|=\sum_{b\in Y_0}|E(X,b)|+\sum_{b\in Y_1}|E(X,b)|<\delta^2|X||Y_0|+\delta|X||Y|\leq (\delta^2+\delta)|X||Y|,
\]
and so $d(X,Y)<\delta^2+\delta\leq\epsilon$. So we may assume that $|Y_1|\geq\delta|Y|$. We apply Proposition \ref{prop:twosticks} with $\alpha'=\beta'=\delta$ and $\delta'=\epsilon'=\delta^2$. In this case $\delta'(1-\beta')= \alpha'(\beta'-\epsilon')$. So we conclude $|X_0|<\delta|X|$ or $|X_0|>(1-\delta)|X|$. If $|X_0|>(1-\delta)|X|$ then $|X_1|<\delta|X|$, and we obtain $d(X,Y)<\epsilon$ by a calculation similar to the case above (where $|Y_1|<\delta|Y|$). Otherwise, if $|X_0|<\delta|X|$ then a similar calculation applied to $\neg E(x,y)$ yields $d(X,Y)>1-\epsilon$. 
\end{proof}

\begin{remark}
The previous arguments also establish a connection between goodness and the notion of excellence (also defined by Malliaris and Shelah \cite{MaShStab}). In particular, a subset $X$ of a finite graph $(V,E)$ is called \emph{$(\epsilon,\delta)$-excellent} if it is $\epsilon$-good and, moreover, for any $\delta$-good set $Y\seq V$, if $X_0=\{a\in X:|E(a,Y)|<\delta|Y|\}$, then either $|X_0|<\epsilon|X|$ or $|X_0|>(1-\epsilon)|X|$. Using Proposition \ref{prop:twosticks}, we can see that if $X\seq V$ is $\epsilon$-good (with $\epsilon<\frac{1}{2}$) then $X$ is $(\frac{\epsilon+2\delta}{1+2\delta},\delta)$-excellent for any $0<\delta<\frac{1}{2}$. So, for example, an $\epsilon$-good set is always $(3\epsilon,\epsilon)$-excellent.

To prove the claim above, suppose $X\seq V$ is $\epsilon$-good. Fix $\delta>0$ and suppose $Y\seq V$ is $\delta$-good. We will apply Proposition \ref{prop:twosticks} with $\alpha=\frac{1}{2}$ and $\beta=\frac{\epsilon+2\delta}{1+2\delta}$. Note $\delta(1-\beta)=\alpha(\beta-\epsilon)$. Let $X_0$ and $Y_1$ be as in the proposition. If $|Y_1|\geq\frac{1}{2}|Y|$ then we conclude $|X_0|<\beta|X|$ or $|X_0|>(1-\beta)|X|$, as desired. So suppose $|Y_1|<\frac{1}{2}|Y|$. Since $X$ is $\epsilon$-good and $Y$ is $\delta$-good, we have
\begin{align*}
X'_0\coloneqq X\backslash X_0=\{a\in X:|\neg E(a,Y)|<\delta|Y|\},\text{ and}\\
Y'_1\coloneqq Y\backslash Y_1=\{b\in Y:|\neg E(X,b)|>(1-\epsilon)|X|\}.
\end{align*}
Since $|Y'_1|>\frac{1}{2}|Y|$, Proposition \ref{prop:twosticks} (applied to $\neg E(x,y)$) yields $|X'_0|<\beta|X|$ or $|X'_0|>(1-\beta)|X|$, i.e., $|X_0|>(1-\beta)|X|$ or $|X_0|<\beta|X|$, as desired.
\end{remark}

\begin{remark}\label{rem:refinegood}
The following simple observation about good sets will be used later in a more complicated context, and thus is worth stating plainly now.

Let $(V,E)$ be a finite graph and  suppose $X\seq V$ is $\epsilon$-good in $(V,E)$. Let $Y$ be a nonempty subset of $X$ and set $\alpha=|Y|/|X|$. Then $Y$ is $\epsilon\alpha\inv$-good in $(V,E)$. Indeed, given $b\in V$, we either have $|E(X,b)|\leq \epsilon|X|$ in which case $|E(Y,b)|\leq |E(X,b)|\leq \epsilon|X|\leq \epsilon\alpha\inv |Y|$, or we have $|\neg E(X,b)|\leq \epsilon|X|$ in which case $|\neg E(Y,b)|\leq \epsilon\alpha\inv|Y|$.
\end{remark}

\subsection{Good decompositions in arbitrary stable graphs}

\begin{theorem}\label{thm:mainUP}
Let $T$ be a complete $\cL$-theory and let $\phi(x,y)$ be an $\cL$-formula defining a graph relation on models of $T$ (so $|x|=|y|=1$). Assume $\phi(x,y)$ is stable in $T$. Fix $M\models T$ and let $\mu$ be a Keisler $\phi$-measure over $M$. Then for all $\epsilon>0$ there is some $m\geq 1$ such that for all $\gamma>0$, there is a partition $M=X_0\cup X_1\cup\ldots\cup X_m$ such that each $X_i$ is $\phi$-definable, $\mu(X_0)<\epsilon$, and if $i\geq 1$ then $X_i$ is $\gamma$-good in $(M,\phi)$ with respect to $\mu$.
\end{theorem}
\begin{proof}
By Theorem \ref{thm:KMstab}, we can write $\mu=\sum_{i\in I}\alpha_ip_i$  where $I$ is an initial segment of $\Z^+$, each $p_i$ is in $S_\phi(M)$, and $\alpha_i=\mu(p_i)$. Note that $\sum_{i\in I}\alpha_i=1$. Now fix some $\epsilon>0$. Then there is some $m\in I$  such that $\sum_{i=1}^m\alpha_i>1-\epsilon$. Now fix some $\gamma>0$.  \medskip

\noindent\emph{Claim.} There is a partition $S_\phi(M)=K_0\cup K_1\cup\ldots\cup K_m$ such that each $K_i$ is clopen, $\mu(K_0)<\epsilon$, and for all $i\geq 1$, $p_i\in K_i$ and $\alpha_i>(1-\gamma)\mu(K_i)$. \medskip

\noindent\emph{Proof.} First, since $S_\phi(M)$ is a Stone space, we can find pairwise disjoint clopen sets $K_1,\ldots,K_m$ with $p_i\in X_i$ for all $i\geq 1$. By Remark \ref{rem:meas-of-type} and definition of $\alpha_i$, we may assume that  $\mu(K_i)<\alpha_i/(1-\gamma)$ for all $i\geq 1$ (without loss of generality, assume $\gamma<1$). Now let $K_0=S_\phi(M)\backslash (K_1\cup\ldots\cup K_m)$. Then $K_0$ is clopen and $K_0,K_1,\ldots,K_m$ partition $S_\phi(M)$. Since $K_0$ is disjoint from $\{p_1,\ldots,p_m\}$, we also have $\mu(K_0)\leq 1-\sum_{i=1}^m\alpha_i<\epsilon$.\claimqed\medskip

Let $K_i=[X_i]$  where $X_i\seq M$ is $\phi$-definable over $M$. Then we have a partition $M=X_0\cup X_1\cup\ldots\cup X_m$. Moreover, $\mu(X_0)<\epsilon$. To finish the proof, we fix $i\geq 1$ and show that $X_i$ is $\gamma$-good in $(M,\phi)$ with respect to $\mu$. So fix $b\in M$. Recall that $p_i$ contains $X_i$. So if $\phi(x,b)\in p_i$ then $\mu(\phi(X_i,b))\geq \mu(p_i)=\alpha_i>(1-\gamma)\mu(X_i)$. Otherwise, if $\neg \phi(x,b)\in p$ then $\mu(\neg\phi(X_i,b))>(1-\gamma)\mu(X_i)$ by the same calculation, and thus $\mu(\phi(X_i,b))<\gamma\mu(X_i)$. 
\end{proof}

\subsection{The pseudofinite normalized counting measure}\label{sec:pcm}

In this section, we briefly summarize one method for expanding an ultraproduct of finite structures to add the pseudofinite normalized counting measure to the language. 

Let $\cL_0$ be a one-sorted first-order language. We define a two-sorted language $\cL$ with sorts $S_h$ and $S_r$, the language $\cL_0$ on sort $S_h$, the language of ordered fields $\{+,\cdot,<,0,1\}$ on $S_r$, and for every $\cL_0$-formula $\phi(x,\ybar)$ with $|x|=1$, a $\ybar$-ary function symbol $f_\phi$ from $S_h^{\ybar}$ to $S_r$. 

Given a finite $\cL_0$-structure $M$, we define an $\cL$-structure $N$ as follows. First, the reduct of $N$ to $S_h$ is $M$, and the reduct of $N$ to $S_r$ is the real field $(\R,+,\cdot,<,0,1)$. Then we interpret each function symbol $f_\phi$ as the map $\bbar\mapsto |\phi(M_0,\bbar)|/|M_0|$. 

Now let $(M_i)_{i\in I}$ be a family of finite $\cL_0$-structures and let $\cU$ be an ultrafilter on $I$. Let $M=\prod_{\cU}M_i$ and $N=\prod_{\cU}N_i$, where $N_i$ is obtained from $M_i$ as above. Note that the reduct of $N$ to the $S_h$ sort is $M$, while the reduct of $N$ to the $S_r$ sort is the ultrapower $\cR\coloneqq (\R,+,\cdot,<,0,1)^\cU$. 

An element $x\in\mathcal{R}$ is \emph{finite} if $|x|<r$ for some $r\in\R$. The \emph{standard part} of a finite element $x\in \mathcal{R}$, denoted $\st(x)$, is  the unique $s\in\R$ such that $|x-s|<\epsilon$ for all $\epsilon\in\R^{>0}$. Given an $\cL_0$-formula $\phi(x,\ybar)$ a tuple $\bbar\in M^{\ybar}$, define $\mu(\phi(x,\bbar))=\st(f_\phi(\bbar))$. It is  straightforward to check that $\mu$ is a finitely additive probability measure on definable subsets of $M$. For example, to check finite additivity of $\mu$, fix  $\phi(x,\ybar)$ and $\psi(x,\ybar')$. Then each $N_i$ satisfies the sentence
\[
\forall \ybar\ybar'\big(f_{\phi\vee\psi}(\ybar\ybar')=f_{\phi}(\ybar)+f_{\psi}(\ybar')-f_{\phi\wedge\psi}(\ybar\ybar')\big).
\]
By {\L}o{\'s}'s Theorem, $N$ also satisfies this. Since the standard part map is a group homomorphism, it then follows that for any $\bbar\in M^{\ybar}$ and $\cbar\in M^{\ybar'}$, we have 
\[
\mu\big(\phi(x,\bbar)\vee\psi(x,\cbar)\big)=\mu\big(\phi(x,\bbar))+\mu(\psi(x,\cbar)\big)-\mu\big(\phi(x,\bbar)\wedge\psi(x,\cbar)\big).
\]

We call $\mu$ the \textbf{pseudofinite normalized counting measure on $M=\prod_{\cU}M_i$}. 

\begin{remark}
The reader familiar with ultralimits may verify that given an $\cL_0$-formula $\phi(x,\ybar)$ and  $\bbar$ from $M=\prod_{\cU}M_i$, if we choose a representative $(\bbar_i)_{i\in I}$ for $\bbar$ then $\mu(\phi(x,\bbar))=\lim_{\cU}|\phi(M_i,\bbar_i)|/|M_i|$.
\end{remark}

\subsection{Stable regularity in for finite graphs}

The goal of this subsection is to prove Theorem \ref{thm:maindown}. First, we start with a lemma that more directly reflects Theorem \ref{thm:mainUP}.

\begin{lemma}\label{lem:maindown}
Fix $k\geq 1$, $\epsilon>0$, and $\sigma\colon\N\to (0,1)$. Then for any finite $k$-stable graph $(V,E)$, there is a partition $V=X_0\cup X_1\cup\ldots\cup X_m$, with $m\leq O_{k,\epsilon,\sigma}(1)$, such that $|X_0|<\epsilon|V|$ and, for all $i\geq 1$, $X_i$ is $\sigma(m)$-good in $(V,E)$.
\end{lemma}
\begin{proof}
Toward a contradiction, suppose the lemma does not hold. Then we have fixed $k\geq 1$, $\epsilon>0$, and $\sigma\colon\N\to (0,1)$ such that for any integer $s\geq 1$, there is a finite $k$-stable graph $(V_s,E_s)$ that does not admit a partition  satisfying the conclusion of the lemma with $m\leq s$. 

 We view each graph $(V_s,E_s)$ as a structure $M_s$ in the  language $\cL_0=\{E\}$ of graphs. With this choice of $\cL_0$, let $\cL$ be the language defined in Section \ref{sec:pcm} and let $N_s$ be the two-sorted $\cL$-structure obtained from $M_s$ as described in that section. Fix a nonprincipal ultrafilter $\cU$ on $\Z^+$, and let $N=\prod_{\cU}N_s$. Let $\mu$ be the pseudofinite normalized counting measure on $M=\prod_{\cU}M_s$. 
 
 In order to have our notation match Section \ref{sec:prelim}, we let $\phi(x,y)$ denote the edge relation $E(x,y)$.
 We view $\mu$ as a Keisler $\phi$-measure by restricting it to $\phi$-formulas.

Note tha{\L}o{\'s}'s Theorem. So we can apply Theorem \ref{thm:mainUP} to $\phi(x,y)$ with respect to the measure $\mu$ and initial parameter $\epsilon$ (fixed at the start of the proof). This yields an integer $m\geq 1$ such that for any $\gamma>0$, we have a partition as in the conclusion of Theorem \ref{thm:mainUP}. Applying this to $\gamma=\sigma(m)$, we obtain $\phi$-definable sets $X_0,X_1,\ldots,X_m$ partitioning $M$ such that $\mu(X_0)<\epsilon$ and, for all $i\geq 1$, $X_i$ is $\gamma$-good in $(M,\phi)$ with respect to $\mu$.

Let $X_i$ be defined by $\psi_i(x,\cbar)$ for some $\cL_0$-formula $\psi_i(x,\zbar)$ (by adding dummy variables we may collect all necessary parameters into a single tuple $\cbar$ from $M$). Let $\theta_i(x;y,\zbar)$ denote $\phi(x,y)\wedge \psi_i(x,\zbar)$. We now define $\cL$-formulas that translate the salient properties of $X_0,X_1,\ldots,X_m$:
\begin{enumerate}
\item[$\ast$] $\chi_1(\zbar)$ says that $\psi_0(x,\zbar),\psi_1(x,\zbar),\ldots,\psi_m(x,\zbar)$ form a partition.
\item[$\ast$] $\chi_2(\zbar)$ says $f_{\psi_0}(\zbar)<\epsilon$.
\item[$\ast$] For $1\leq i\leq m$, $\chi^i_3(\zbar)$ says 
\[
\forall y\bigg(\big(f_{\theta_i}(y,\zbar)< \gamma f_{\psi_i}(\zbar)\big) \vee \big(f_{\theta_i}(y,\zbar)> (1-\gamma)f_{\psi_i}(\zbar)\big)\bigg).
\] 
\end{enumerate}
Let $\chi(\zbar)$ denote $\chi_1(\zbar)\wedge \chi_2(\zbar)\wedge\bigwedge_{i=1}^m\chi^{i}_3(\zbar)$. Then  $N\models\chi(\cbar)$. Let $\cbar=[(\cbar_s)_{s\geq 1}]_{\cU}$ where $\cbar_s$ is a tuple from $V_s$. Let $I=\{s\geq 1:N_s\models\chi(\cbar_s)\}$, and note that $I\in\cU$ by {\L}o{\'s}'s Theorem. Since $\cU$ is nonprincipal, there is some $s\in I$ such that $s\geq m$. 

Let $X^s_i=\psi_i(V_s,\cbar_s)$. Since $N_s\models\chi(\cbar_s)$, we have that $X^s_0,X^s_1,\ldots,X^s_m$ partition $V_s$, $|X^s_0|<\epsilon|V_s|$, and if $i\geq 1$ then $X^s_i$ is $\sigma(m)$-good in $(V_s,E_s)$. Since $s\geq m$, this contradicts the initial choice of $(V_s,E_s)$.
\end{proof}

Finally, we use the previous lemma to prove Theorem \ref{thm:maindown}. At this point, the proof becomes entirely finitary and combinatorial. We make use of Lemma \ref{lem:twosticks}, and also illustrate the leverage obtained by using functional error.

\begin{proof}[\textnormal{\textbf{Proof of Theorem \ref{thm:maindown}}}]
From the statement of theorem, we have fixed $k\geq 1$, $\epsilon>0$, and $\sigma\colon\N\to (0,1)$. Without loss of generality, assume $\sigma$ is decreasing. Define $\tau\colon \N\to (0,1)$ such that $\tau(m)=\epsilon\sigma(\lfloor 2m^2\epsilon\inv\rfloor)^2/8m$. Now let $(V,E)$ be a finite $k$-stable graph. We apply Lemma \ref{lem:maindown} to $(V,E)$ with parameters $k$, $\epsilon/2$, and $\tau$. This yields a partition $V=X_0\cup X_1\cup\ldots\cup X_m$, with $m\leq O_{k,\epsilon,\sigma}(1)$, such that $|X_0|<\epsilon|V|/2$ and, for all $i\geq 1$, $X_i$ is $\tau(m)$-good in $(V,E)$. Before applying Lemma \ref{lem:twosticks} to obtain almost homogeneous pairs, we first refine $X_0,X_1,\ldots,X_m$ into an equipartition. 

For each $1\leq i\leq m$, partition $X_i=X_{i,1}\cup\ldots\cup X_{i,t_i}\cup W_i$ so that $|X_{i,j}|=\lceil\frac{\epsilon}{2m}|V|\rceil$ for all $1\leq j\leq t_i$, and $|W_i|<\frac{\epsilon}{2m}|V|$ (we allow $t_i=0$ if $|X_i|<\frac{\epsilon}{2m}|V|$). Let $Y_1,\ldots,Y_n$ be an enumeration of $\{X_{i,j}:1\leq i\leq m,~1\leq j\leq t_i\}$, and set $Y_0=X_0\cup \bigcup_{i=1}^m W_i$. Then we have a partition $V=Y_0\cup Y_1\cup\ldots\cup Y_n$, which we will now show satisfies the conclusion of the theorem.

First we show $n\leq O_{k,\epsilon,\sigma}(1)$. For each $1\leq i\leq m$, we have
\[
|V|\geq|X_i|\geq\textstyle t_i\frac{\epsilon}{2m}|V|,
\]
and so $t_i\leq 2m\epsilon\inv$. Therefore $n= \sum_{i=1}^m t_i\leq 2m^2\epsilon\inv\leq O_{k,\epsilon,\sigma}(1)$.

Next, note that 
\[
\textstyle|Y_0|=|X_0|+\sum_{i=1}^m |W_i|<\frac{\epsilon}{2}|V|+m\frac{\epsilon}{2m}|V|=\epsilon|V|.
\]
Also note that if $i,j\geq 1$ then $|Y_i|=\lceil \frac{\epsilon}{2m}|V|\rceil=|Y_j|$ by construction. 

Finally, we show that if $1\leq i,j\leq n$ then either $d(Y_i,Y_j)<\sigma(n)$ or $d(Y_i,Y_j)>1-\sigma(n)$. By Lemma \ref{lem:twosticks}, it suffices to show that each $Y_t$ is $\sigma(n)^2/4$-good in $(V,E)$. So fix some $1\leq t\leq n$. Then $Y_t=X_{i,j}$ for some $1\leq i\leq m$ and $1\leq j\leq t_i$. Recall that $X_i$ is $\tau(m)$-good in $(V_s,E_s)$ and $|X_{i,j}|\geq \frac{\epsilon}{2m}|V|\geq \frac{\epsilon}{2m}|X_i|$. So by choice of $\tau$ and Remark \ref{rem:refinegood}, $X_{i,j}$ is $\sigma(\lfloor 2m^2\epsilon\inv\rfloor)^2/4$-good in $(V,E)$. Since $\sigma$ is decreasing and $n\leq \lfloor 2m^2\epsilon\inv\rfloor$, it follows that $X_{i,j}$ is $\sigma(n)^2/4$-good in $(V,E)$, as desired.
\end{proof}

\section{Further remarks}\label{sec:more}

\subsection{The Malliaris-Pillay structure theorem for stable formulas}\label{sec:MP}

In this section, we discuss the variation of  Theorem \ref{thm:mainUP} from \cite{MallPi}, which concerns  regularity for arbitrary stable formulas with respect to an arbitrary Keisler measure. In this context,  the most direct analogue of Theorem \ref{thm:maindown} would require a ``binary" measure on $\phi(x,y)$ in order to work with the notion of edge density $d(X,Y)$. Such an analogue can be obtained using the fact that Keisler measures obtained from stable formulas admit well-defined ``Morley products" satisfying Fubini; see \cite[Section 2.5]{ChStNIP} and \cite{GanStab}. Malliaris and Pillay take a different approach, and work with a uniform variation on the notion of good pairs. To motivate this, let us summarize the various (mostly equivalent) forms of ``strong regularity" for pairs of subsets in a finite graph.

Let $(V,E)$ be a finite graph and let $X,Y\seq V$ be nonempty sets. We define/recall some terminology.
\begin{enumerate}
\item[$(i)$] $(X,Y)$ is \emph{$\epsilon$-homogeneous} if $d(X,Y)<\epsilon$ or $d(X,Y)>1-\epsilon$.
\item[$(ii)$] $(X,Y)$ is \emph{$\epsilon$-good} if for all $a\in X$, $|E(a,Y)|<\epsilon|Y|$ or $|E(a,Y)|>(1-\epsilon)|Y|$, and for all $b\in Y$, $|E(X,b)|<\epsilon|X|$ or $|E(X,b)|>(1-\epsilon)|X|$.
\item[$(iii)$] $(X,Y)$ is \emph{$\epsilon$-special} if there are $X'\seq X$ and $Y'\seq Y$ such that $|X'|>(1-\epsilon)|X|$, $|Y'|>(1-\epsilon)|Y|$, and either:
\begin{enumerate}
\item[$\ast$] for all $a\in X'$ and $b\in Y'$, $|E(a,Y)|<\epsilon|Y|$ and $|E(X,b)|<\epsilon|X|$, or
\item[$\ast$] for all $a\in X'$ and $b\in Y'$, $|E(a,Y)|>(1-\epsilon)|Y|$ and $|E(X,b)|>(1-\epsilon)|X|$.
\end{enumerate}
\end{enumerate}
Using calculations similar to the proof of Lemma \ref{lem:twosticks}, one can establish the following relationships between these notions. (Note that item (3) is Lemma \ref{lem:twosticks}.)
\begin{enumerate}
\item[(1)] If $(X,Y)$ is $\epsilon$-homogeneous then it is $\sqrt{\epsilon}$-special.
\item[(2)] If $(X,Y)$ is $\epsilon$-special then it is $2\epsilon$-homogeneous.
\item[(3)] If $(X,Y)$ is $\epsilon$-good then it is $2\sqrt{\epsilon}$-homogeneous.
\end{enumerate}
There is obviously no similar implication from homogeneous/special to good. However, one could weaken condition $(ii)$ so that, as in $(iii)$, the universal quantifiers are over large sets $X'\seq X$ and $Y'\seq Y$; call this ``almost $\epsilon$-good". Then clearly any $\epsilon$-special pair is almost $\epsilon$-good. Moreover, is not hard to adjust the proof of Lemma \ref{lem:twosticks} to show that an almost $\epsilon$-good pair is $\epsilon'$-homogeneous for appropriately chosen $\epsilon'$ (in fact, the ``symmetry lemma" given in \cite[Lemma 5.9]{TWgraphs} is phrased using almost good pairs). 

The version of stable regularity proved by Malliaris and Pillay \cite{MallPi} is stated in terms of special\footnote{We note that the terminology ``special" does not appear in \cite{MallPi}. We have chosen to use it here for  easier exposition (and for lack of a better term).}  pairs, but of course in the context of arbitrary Keisler measures. Note that a general formula $\phi(x,y)$  is more naturally treated as a bipartite graph, and indeed most model-theoretic regularity results are given in the bipartite setting (e.g., \cite{MallPi,ChSt,ChStNIP,CCP}).  This necessitates working with two measures, one for $\phi$ and the other for $\phi^*$. So altogether we need to work with the following data: a model $M\models T$, a formula $\phi(x,y)$, a Keisler $\phi$-measure $\mu$, a Keisler $\phi^*$-measure $\nu$, a $\phi$-definable set $X\seq M^x$, and a $\phi^*$-definable set $Y\seq M^y$. Then we say that the pair $(X,Y)$ is \emph{$\epsilon$-special with respect to $\mu$ and $\nu$} if the obvious analogue of condition $(iii)$ holds, with the extra assumption that $X'$ is $\phi$-definable and $Y'$ is $\phi^*$-definable.

We  now state and prove what is essentially the main result from \cite{MallPi}. We only make one  improvement, namely, we  adjust the order of parameters in order to accommodate the ``functional error" regime (specifically, the parameter $\gamma$ is allowed to depend on the sizes of the partitions of $M^x$ and $M^y$).

\begin{theorem}\label{thm:MP}
Let $T$ be a first-order $\cL$-theory and let $\phi(x,y)$ be an $\cL$-formula. Assume $\phi(x,y)$ is stable in $T$ and fix a model $M\models T$. Let $\mu$ be a Keisler $\phi$-measure over $M$, and $\nu$ a Keisler $\phi^*$-measure over $M$. Then for all $\epsilon>0$, there are $m,n\geq 1$ such that for all $\gamma>0$, there are partitions $M^x=X_0\cup X_1\cup\ldots\cup X_n$ and $M^y=Y_0\cup Y_1\cup\ldots\cup Y_n$ such that each $X_i$ is $\phi$-definable, each $Y_i$ is $\phi^*$-definable, $\mu(X_0)<\epsilon$, $\mu(Y_0)<\epsilon$, and if $i,j\geq 1$ then $(X_i,X_j)$ is $\gamma$-special with respect to $\mu$ and $\nu$. 
\end{theorem}
\begin{proof}
We start with the same argument as in the proof of Theorem \ref{thm:mainUP}, but applied to both $\mu$ and $\nu$. After fixing $\epsilon>0$ and applying the same claim, this yields a partition $M^x=X_0\cup X_1\cup\ldots \cup X_m$ into $\phi$-definable sets and a partition $M^y=Y_0\cup Y_1\cup\ldots\cup Y_n$ into $\phi^*$-definable sets such $\mu(X_0)<\epsilon$, $\nu(Y_0)<\epsilon$ and for $i,j\geq 1$, $X_i$ contains a type $p_i\in S_\phi(M)$ with $\mu(p_i)>(1-\gamma)\mu(X_i)$, and $Y_j$ contains a type $q_j\in S_{\phi^*}(M)$ with $\nu(q_j)>(1-\gamma)\nu(Y_j)$. 

Now fix $i,j\geq 1$. We show that $(X_i,X_j)$ is $\gamma$-special with respect to $\mu$ and $\nu$. To streamline the argument, we use the following notation. Given an arbitrary formula $\psi$, let $\psi^1$ and $\psi^0$ denote $\psi$ and $\neg\psi$, respectively.

 Let $\psi(y)$ be a $\phi^*$-formula over $M$ defining $p_i$, and let $\theta(x)$ be a $\phi$-formula over $M$ defining $q_j$. Choose $t\in\{0,1\}$ so that $\psi^t(y)\in q_j$. By Harrington's Lemma (see Remark \ref{rem:HL} below), we also have $\theta^t(x)\in p_i$. Let $X'=\theta^t(X_i)$ and $Y'=\psi^t(Y_j)$. Then $X'$ is in $p_i$ so $\mu(X')\geq\mu(p_i)>(1-\gamma)\mu(X_i)$, and $Y'$ is in $q_j$ so $\nu(Y')\geq\nu(q_j)>(1-\gamma)\nu(Y_j)$. Finally, suppose $b\in Y'$. Then $M\models\psi^t(b)$, and so $\phi^t(x,b)\in p_i$. So $p_i$ contains $\phi^t(X_i,b)$, and thus $\mu(\phi^t(X_i,b))>(1-\gamma)\mu(X_i)$ as above. Similarly, we have $\nu(\phi^t(a,Y_j))>(1-\gamma)\nu(Y_j)$ for all $a\in X'$. Altogether, $X$ and $Y'$ witness that $(X_i,Y_j)$ is $\gamma$-special with respect to $\mu$ and $\nu$.
\end{proof}

\begin{remark}\label{rem:HL}
We recall \emph{Harrington's Lemma}, which was used in the proof of the previous theorem. Suppose $\phi(x,y)$ is stable in $T$ and fix $M\models T$. Given $p\in S_\phi(M)$ and $q\in S_{\phi^*}(M)$, let $\psi(y)$ be a $\phi^*$-formula over $M$ defining $p$ and let $\theta(x)$ be a $\phi$-formula over $M$ defining $q$. Then Harrington's Lemma asserts that $\psi(y)\in q$ if and only if $\theta(x)\in p$. Indeed, suppose this fails and (without loss of generality), assume we have $\neg\psi(y)\in q$ and $\theta(x)\in p$. Inductively construct sequences $(a_i)_{i<\omega}$ from $M^x$ and $(b_i)_{i<\omega}$ from $M^y$ such that
\[
a_{i+1}\models \{\theta(x)\}\cup p|_{b_1,\ldots,b_i}\text{ and }b_{i+1}\models\{\neg\psi(y)\}\cup q|_{a_1,\ldots,a_{i+1}}.
\]
Then by construction, we have $M\models\phi(a_i,b_j)$ if and only if $i\leq j$, which is a contradiction.
\end{remark}

\subsection{Removing the exceptional set in groups}\label{sec:finsum}

Recall that all of our previous regularity results involve irregular error controlled by a function of the size of the partition (rather than  constant error $\epsilon>0$). The cost of doing this is that we need to allow for an exceptional set of vertices of small constant measure. This exceptional set originates in the proof of Theorem \ref{thm:mainUP} when approximating a Keisler $\phi$-measure (given by a possibly infinite weighted sum of types) by a finite sum of types. Therefore, if one were to somehow know that the initial Keisler $\phi$-measure were already a finite sum of types, then the exceptional set would no longer be necessary. This raises an interesting question of when a family of finite (stable) graphs has the property that, for any ultraproduct of members in the family, the pseudofinite normalized counting measure  is a \emph{finite} sum of types.  We will now briefly discuss an important example in which this is the case, namely, stable subsets of groups.

Let $G$ be a group and suppose $A\seq G$ is an arbitrary subset. Then we say that $A$ is \textbf{stable in $T$} if the formula $A(x\cdot y)$ is stable in $\Th(G,\cdot,A)$. In \cite{TeWo,TeWo2},  the second author and Wolf proved an  ``arithmetic" regularity lemma for stable subsets of finite abelian groups. When compared to Green's \cite{GreenSLAG} general arithmetic regularity lemma for arbitrary subsets of finite abelian groups, their result yields stronger conclusions for stable sets analogous to the comparison between Malliaris-Shelah and Szemer\'{e}di. In the interim between \cite{TeWo} and \cite{TeWo2}, a pseudofinite proof for stable subsets of arbitrary finite groups was proved by the authors and Pillay \cite{CPT}. Let us now state (but not prove) that result.

\begin{theorem}
Suppose $G$ is a finite group and $A\seq G$ is $k$-stable. Then for any function $\sigma\colon \N\to (0,1)$, there is a normal subgroup $H\leq G$ of index $n\leq O_{k,\sigma}(1)$, such that for any $g\in G$, either $|A\cap gH|<\sigma(n)|H|$ or $|A\cap gH|>(1-\sigma(n))|H|$.
\end{theorem}

Note that in the previous theorem the subgroup $H$ provides a partition of $G$ into cosets, and the conclusion implies almost homogeneity  for all pairs of cosets with respect to the (bipartite) graph $x\in yA$ (see \cite[Proposiiton 3.4]{CPT}). Thus in this case there is no need for the exception set of constant size. The reason is that, after moving to the suitable model-theoretic context, the relevant Keisler measure is a \emph{finite} sum of types. In particular,  if $G$ is a first-order expansion of a group and $\phi(x,\ybar)$ is a left-invariant stable formula (with $x$ a singleton), then the set $X\seq S_\phi(G)$ of generic types is finite and nonempty. Moreover, there is a unique left-invariant Keisler $\phi$-measure over $G$, which is precisely the average over $X$.  We refer the reader to \cite{CPT} for the relevant definitions, and the proof of the above theorem.\footnote{The result is not stated with functional error in \cite{CPT}, however functional error is evident from (and used explicitly in) the proof.} As usual, the pseudofinite proof provides no explicit bounds, unlike the results from \cite{TeWo,TeWo2} in the abelian case. A finitary proof for arbitrary finite groups, with quantitative bounds improving on the abelian case, was later given by the first author in \cite{CoQSAR}.

\end{document}